\documentclass[graybox]{svmult}

\usepackage{mathptmx}       
\usepackage{helvet}         
\usepackage{courier}        
\usepackage{type1cm}        
\usepackage{makeidx}         
\usepackage{graphicx}        
\usepackage{multicol}        
\usepackage[bottom]{footmisc}

\usepackage{amsmath}

\usepackage{amssymb}

\usepackage[latin1]{inputenc}
\usepackage{eurosym}
\usepackage{hyperref}
\usepackage{dsfont}
\usepackage{verbatim}
\usepackage{subfigure}

\usepackage[displaymath,mathlines]{lineno}

\allowdisplaybreaks

\usepackage{hyperref}

\usepackage{ifthen}

\makeindex

\renewcommand{\div}{\operatorname{div}}

\newcommand{\Tt}{{\mathbb{T}}}

\newcommand{\cqd}{\hfill $\blacksquare$}

\newcommand{\epsi}{\varepsilon}

\def\geq{\geqslant}

\def\appendixname{\empty}

\makeindex             

\begin{document}

\title*{Conservation laws arising in the study of forward-forward Mean-Field Games.}
\author{Diogo Gomes, Levon Nurbekyan, and Marc Sedjro}
\institute{Levon Nurbekyan \at CEMSE Division, King Abdullah University of Science and Technology (KAUST), Thuwal 23955-6900. Saudi Arabia. \\
	\email{levon.nurbekyan@kaust.edu.sa},\and Diogo Gomes \at CEMSE
	Division, King Abdullah University of Science and Technology (KAUST),  Thuwal 23955-6900. Saudi Arabia.\\ \email{diogo.gomes@kaust.edu.sa}
\and Marc Sedjro \at  CEMSE
Division, King Abdullah University of Science and Technology (KAUST),Thuwal 23955-6900. Saudi Arabia.\\
\email{marc.sedjro@kaust.edu.sa}}

%
%
\maketitle

\abstract{We consider forward-forward Mean Field Game (MFG) models that  arise in numerical approximations of stationary MFGs. First, we establish a link between these models and a class of hyperbolic conservation laws as well as certain nonlinear wave equations. Second, we investigate existence and long-time behavior of solutions for such models.}

\section{Introduction}
\label{sec:1}

A few years ago,  Lasry and Lions \cite{ll3} and Caines, Huang and Malhame \cite{Caines1}  independently introduced the Mean-Field Game (MFG) framework. These games model competitive interactions in a population of agents with a dynamics given by an optimal control problem.  A typical                                                                                              MFG is determined by the system
\begin{eqnarray}\label{eq: General MFG}
\begin{cases}
-u_t + H(x, Du)= \epsi\Delta u+ g[m]\qquad\qquad \quad \mathbb{ T}^d\times [0,T] \\
m_t-\div(D_p H(x, Du(x) m))=\epsi \Delta m \qquad\quad \mathbb{T}^d\times [0,T]. 
\end{cases}
\end{eqnarray}
Here, $\mathbb{T}^d$ is the $d$-dimensional torus and the Hamiltonian, $H$, the coupling, $g$, and the terminal time, $T>0$, are prescribed. The first equation in (\ref{eq: General MFG}) is a Hamilton-Jacobi equation. This equation states the optimality of the value function,  $u$, associated with the control problem. The second equation is the Fokker-Planck equation that determines the evolution of the density of the agents, $m$.  $\epsi\geq 0$ is a viscosity parameter. If $\epsi=0$, we refer  to \eqref{eq: General MFG} as a first-order MFG. Otherwise, we refer to \eqref{eq: General MFG} as  a parabolic MFG.
 System  \eqref{eq: General MFG} is typically complemented with initial-terminal conditions:
\begin{equation}
\begin{cases}
u(x,T) = u_T(x) \\
m(x,0) = m_0 (x).
\end{cases}
\end{equation}
Extensive research has been conducted in the study of MFGs. For the parabolic problem, strong and weak solutions were, respectively, examined in \cite{GPim2, GPim1, ll2} and \cite{ ll2, porretta2}. The stationary problem for the parabolic case has also generated great interest - several results on the existence of classical and weak solutions were obtained in \cite{GMit,GP,GPat,GPatVrt}. The uniqueness of a solution in all these cases relies on the monotonicity of $g$.

Here, we consider a related problem, forward-forward MFG, that is derived from (\ref{eq: General MFG}) by reversing the time in the Hamilton-Jacobi equation. Accordingly, we consider the system 
\begin{equation}\label{eq:FF MFG}
\begin{cases}
u_t + H(x, Du)= \epsi\Delta u+ g[m] \\
m_t-\div(D_p H(x, Du(x)) m)=\epsi \Delta m.
\end{cases}
\end{equation}
Because of the time reversal, we prescribe initial-initial conditions:
\begin{eqnarray}\label{eq: ff initial data}
\begin{cases}
u(x,0) = u_0(x) \\
m(x,0) = m_0 (x)
\end{cases}
\end{eqnarray}
 for \eqref{eq:FF MFG}.
Forward-forward models were first introduced  in  \cite{DY} to numerically approximate solutions of stationary MFGs.  Before our contributions \cite{GomNurSed16}, no rigorous results on the long-time convergence of forward-forward MFGs had been proven.
Additionally, the forward-forward problem is interesting on its own right  as a learning game. In a standard MFG, agents follow  optimal trajectories of a terminal-value optimal control problem. For the forward-forward problem, only initial data is given. Thus, only past optimal trajectories are relevant to the density evolution.  Accordingly, the density evolution feeds on past information of the density.

This paper complements the results in  \cite{GomNurSed16} by examining several cases that can be studied explicitly. In Section \ref{sec:2}, 
we consider linear Hamiltonians and show that the wave equation is a special case of the forward-forward model. In Section \ref{sec:3}, we study quadratic forward-forward MFGs using elementary conservation law techniques. In particular, we compute Riemann invariants and characterize invariant domains. We end the paper by recalling the main result from   \cite{GomNurSed16}  on the convergence of 
forward-forward MFGs.

The study of forward-forward MFG presents substantial challenges even in dimension one which we consider here. The first-order  forward-forward problem can be rewritten as a nonlinear wave equation that inherits the non-linearity of the Hamiltonian. For quadratic Hamiltonian, the system reduces  to elastodynamics equation. In general, the forward-forward problem can be rewritten, formally, as a system of one-dimensional conservation law. This reformulation allows us to use methods and ideas from the theory of conservation laws such as  hyperbolicity, genuinely nonlinearity, Riemann invariants and invariant domains.

For the parabolic forward-forward problem, standard techniques yield existence and uniqueness of a solution. Here, we investigate the long-time convergence of this solution to the solution of a stationary MFG.

\section{ First-order, one-dimensional, forward-forward Mean-Field Games as nonlinear wave equations}
\label{sec:2}

In this section, we consider the first-order, one-dimensional, forward-forward MFG:
\begin{eqnarray}\label{eq: foff mfg}
\begin{cases}
u_t+H(u_x)=g(m),\\
m_t-(mH'(u_x))_x=0.
\end{cases}
\end{eqnarray}
  To gain insight into the system above, we consider two simple examples: linear and quadratic Hamiltonians.
  First, we
  assume that $H$ is linear, that is, $H(p)=p$, and that the coupling, $g$, is smooth invertible with $g'\neq 0$. In this case, $u$ satisfies the wave equation: 
  \begin{equation}\label{eq : wave equations}
 u_{tt}-u_{xx} =0.
 \end{equation}
  Thus, for smooth initial data in  $\eqref{eq: ff initial data}$, solutions of (\ref{eq: foff mfg}) are: 
  \begin{equation}\label{eq :ff linear sol u }
 u(x,t) =u_0(x-t)+\dfrac{1}{2}\int_{x-t}^{x+t} g(m_0(s)) ds,
  \end{equation}
 and 
   \begin{equation}\label{eq :ff linear sol m }
 m(x,t) =m_0(x-t). 
 \end{equation}
 
Next, we  
assume that $H$ is quadratic, $H(p)=p^2/2$, and that $g$ is logarithmic, $g(m)=\ln(m)$. Then, after elementary computations, we obtain that $u$ satisfies the nonlinear wave  equation 
 \begin{equation}\label{ nonlinear elasto}
 u_{tt}-(1+u_x^2)u_{xx} =0.
 \end{equation}
  This nonlinear equation is known in elastodynamics  and, in Lagrangian coordinates, it is a system of hyperbolic conservation laws.
 \begin{equation} \label{eq : elastodynamics}
 \begin{cases}
 v_t -w_x= 0 \\
 w_t - \sigma(v)_x=0
 \end{cases}
 \end{equation}
  with
 $$ w=u_t,\quad v=u_x,\quad \hbox{and}\quad\sigma(z)=z+\frac{z^3}{3}.$$ 
The system \eqref{eq : elastodynamics}  falls within a class of conservation laws   investigated in \cite{DiPerna83}, in the whole space, and in \cite{DeStTz00}, in the periodic case.

\section{ One-dimensional forward-forward Mean-field Games as conservation laws}
\label{sec:3}
In this section, we discuss  how certain one-dimensional forward-forward MFGs can be written as a system of one-dimensional conservation laws. Furthermore, we analyze latter and compute the corresponding Riemann invariants. 
For simplicity, we consider  the forward-forward problem with quadratic Hamiltonian and a quadratic coupling:
\begin{eqnarray}\label{eq: foff mfg quadratic}
\begin{cases}
u_t+u_x^2/2=m^2/2,\\
m_t-(mu_x)_x=0.
\end{cases}
\end{eqnarray}
We complement \eqref{eq: foff mfg quadratic} with initial-initial condition
 \begin{eqnarray}\label{eq: ff initial data1}
 \begin{cases}
 u(x,0) = u_0(x) \\
 m(x,0) = m_0 (x) >0.
 \end{cases}
  \end{eqnarray}
   We note that the Fokker-Planck equation preserves positivity. As such, the density $m(t,\cdot)$ is positive for all $t>0$. We formally differentiate the first equation  with respect to $x$  and then set $v=u_x$. As a result, we  obtain

\begin{eqnarray}\label{eq: con laws quadratic}
\begin{cases}
v_t+\left( v^2/2-m^2/2\right)_x =0,\\
m_t-(mv)_x=0.
\end{cases}
\end{eqnarray}
The  associated flux function to the system of conservation laws \eqref{eq: con laws quadratic} is given by 
\begin{equation}\label{eq: nonlinear function}
F(v,m)= (v^2/2-m^2/2, -vm)\qquad m>0, \quad v\in \mathbb{R}.
\end{equation}
Finally, 
we compute its Jacobian and get
\begin{equation}\label{eq: Jacobian nonlinear function}
DF(v,m)= \begin{bmatrix}
v & -m\\
-m & -v
\end{bmatrix}.
\end{equation}

\subsection{Hyperbolicity and Genuine Nonlinearity}
A simple computation shows that (\ref{eq: Jacobian nonlinear function}) has two distinct eigenvalues, $\lambda_1$ and  $\lambda_2$,  given by

\begin{equation}\label{eq: Eigenvalues}
\lambda_1= -\sqrt{v^2  + m^2}\quad \hbox{and}\quad \lambda_2= \sqrt{v^2  + m^2},
\end{equation}
with  respective eigenvectors given by
\begin{equation}\label{eq: Eigenvectors}
r_1=\begin{bmatrix}
-v+ \sqrt{v^2  + m^2}  \\
m
\end{bmatrix}   \quad \hbox{and}\quad r_2= \begin{bmatrix}
v+ \sqrt{v^2  + m^2} \\
-m
\end{bmatrix}.
\end{equation}
From the discussion above, the system of conservation laws in (\ref{eq: con laws quadratic}) is strictly hyperbolic. Note that 
\begin{equation}\label{eq:genuine non lin 1}
\nabla \lambda _1\cdot r_1 = \dfrac{-m^2+ v \left(v-\sqrt{v^2  + m^2} \right) }{m\sqrt{v^2  + m^2}}
\end{equation}
and 
\begin{equation}\label{eq:genuine non lin 2}
\nabla \lambda _2\cdot r_2 = \dfrac{m^2-v \left(v+\sqrt{v^2  + m^2} \right) }{m\sqrt{v^2  + m^2}}.
\end{equation}
Observe that 
\begin{equation}\label{eq:genuine non lin 2}
\nabla \lambda _i\cdot r_i= 0 \Longleftrightarrow m^2-v \left(v+\sqrt{v^2  + m^2} \right) =0\qquad i=1,2.
\end{equation}
As a result,	(\ref{eq: con laws quadratic}) is a strictly hyperbolic genuinely nonlinear  system  outside the set $\mathcal{S}$ given by 
\begin{equation}\label{eq: singular set}
\mathcal{S}:= \{ (v,m): m^2=3v^2,\; m>0\}.
\end{equation}

\subsection{Riemann invariants and invariant domains}

In the following proposition, we provide an explicit expression for Riemann invariants for the system of conservation laws in the quadratic case. As a consequence, we obtain invariant sets for the corresponding problem with viscosity.
\begin{proposition}
	The system of conservation laws (\ref{eq: foff mfg quadratic})  
	has the following Riemann invariants
	$$w_1(v, m) = \sqrt{(m^2 + v^2)^3 }- v^3 +3 v m^2$$
	and
	$$ w_2(v, m) = \sqrt{(m^2 + v^2)^3}+ v^3 - 3 v m^2,$$
	corresponding to the eigenvectors $r_1$ and $r_2$.	
\end{proposition}
\begin{proof}
Note that 
 $w_i$ is  such that $\nabla w_i$ is parallel to the eigenvector $r_i$. This means that $w_1$ solves
\begin{equation}\label{eq: 1-rieman invariant 0}
(v+\sqrt{m^2+v^2})\partial_v w_1- m\partial_m w_1= 0.
\end{equation}
In a similar way,  for  $w_2$,
\begin{equation}\label{eq: 2-rieman invariant 0}
(-v+\sqrt{m^2+v^2})\partial_v w_1+ m\partial_m w_1= 0.
\end{equation}
\cqd
\end{proof}

Using Riemann invariants, we can identify invariant domains for the viscosity solutions to \eqref{eq: con laws quadratic}. These are obtained by looking at level curves of $w_1$ and $w_2$ see Fig. 1.

\begin{figure}
	\centering
	\subfigure[Level sets of $w_1$.]{
		\includegraphics[width=5cm,height=5cm]{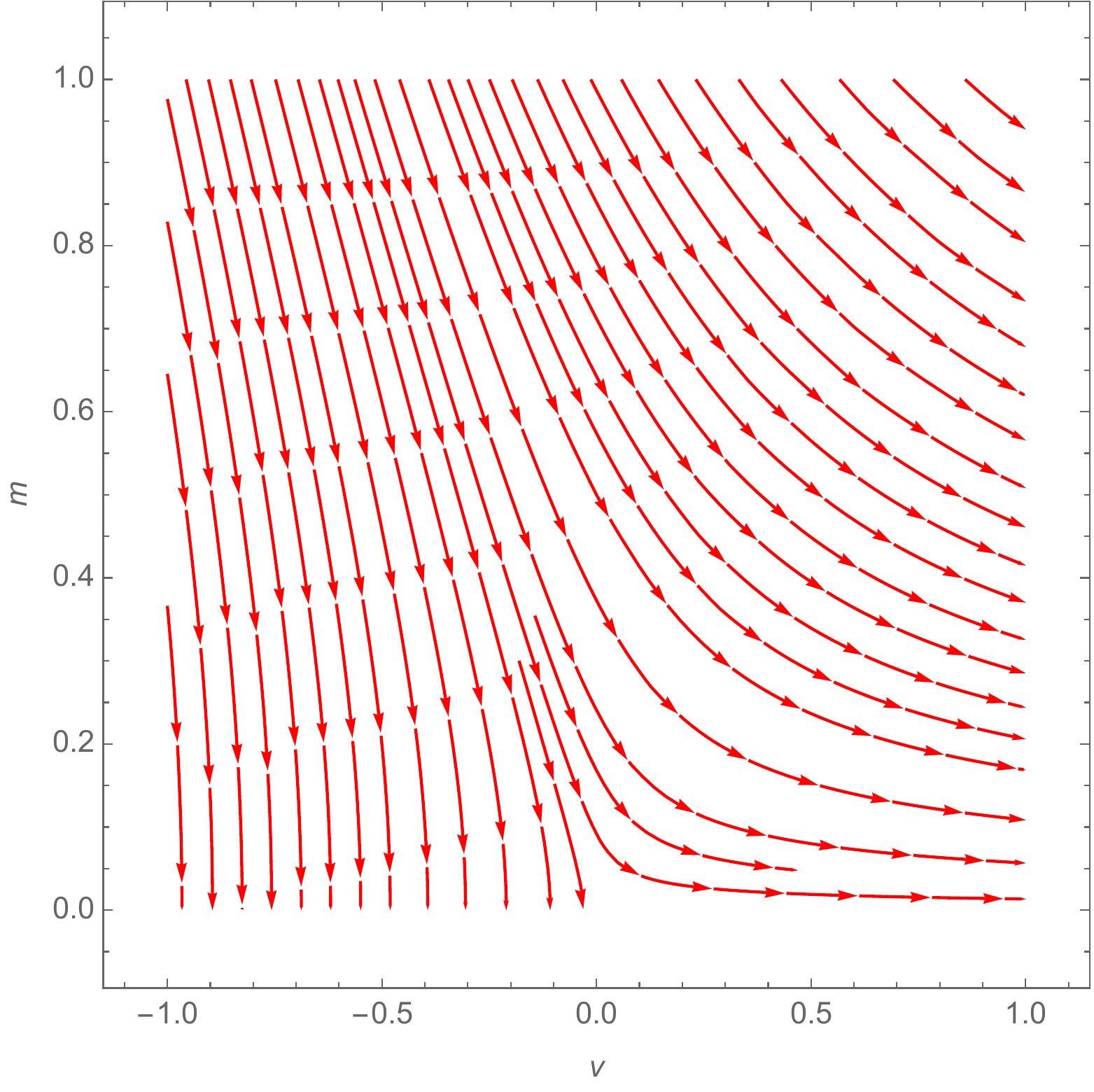}
	} \quad
	\subfigure[Level sets of $w_2$.]{
	\includegraphics[width=5cm,height=5cm]{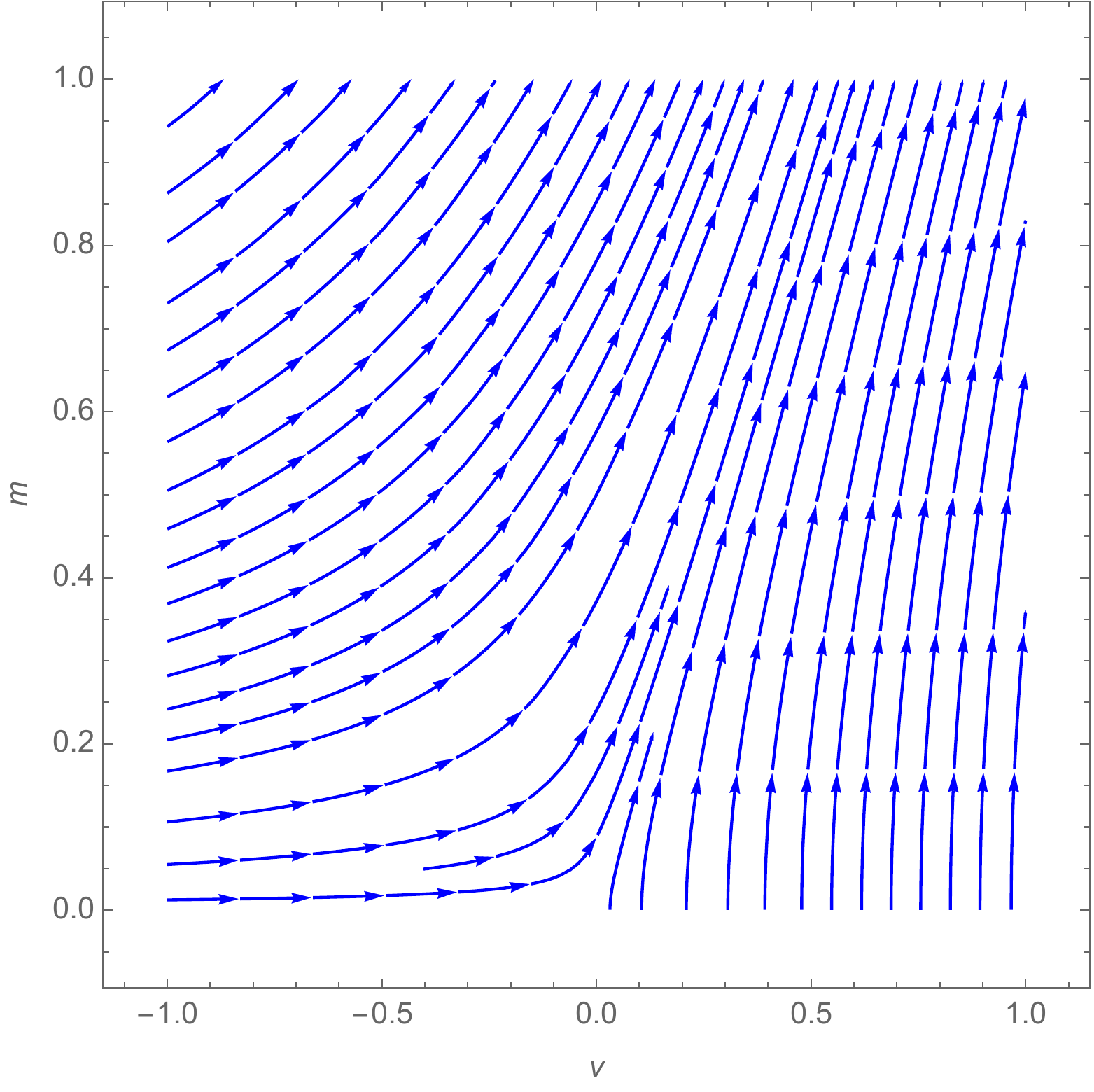}
    }
	\caption{Invariant domains associated with the hyperbolic conservation laws \ref{eq: con laws quadratic}.}
\end{figure}

\section{Convergence of one-dimensional, forward-forward, parabolic conservation laws}
Here, we consider \eqref{eq: General MFG} in dimension 1. As before, by differentiating the first equation  
 with respect to $x$ and setting  $v=u_x$, we get
\begin{equation}\label{eq: parabolic ffmfg}
\begin{cases}
v_t+(v^2/2-m^2/2)_x=\epsi v_{xx},\\
m_t-(mv)_x= \epsi m_{xx}.
\end{cases}
\end{equation}
The system  \eqref{eq: parabolic ffmfg} has a  unique local smooth solution for bounded initial data. Now, we investigate the long-time convergence of the solution. For that, we, additionally, require
\begin{equation}\label{eq: fixedmeanvalues}
\int\limits_{\Tt} v(x,0)dx=0, \quad \int\limits_{\Tt} m(x,0)dx=1,
\end{equation}
which are natural assumptions from the perspective of periodic MFGs. The following theorem is proven in \cite{GomNurSed16}.
\begin{theorem}
	If $v,m \in C^2(\Tt \times (0,+\infty))\cap C(\Tt \times
	[0,+\infty)),\ m>0,$ solve \eqref{eq: parabolic ffmfg} and satisfy \eqref{eq: fixedmeanvalues} then, we have that
	\begin{equation}\label{eq: longtime l1}
	\lim \limits_{t\to \infty} \int\limits_{\Tt} |v(x,t)|dx=0,\quad \lim \limits_{t\to \infty} \int\limits_{\Tt} |m(x,t)-1|dx=0.
	\end{equation}
\end{theorem}

\bibliographystyle{plain}

\def\polhk#1{\setbox0=\hbox{#1}{\ooalign{\hidewidth
			\lower1.5ex\hbox{`}\hidewidth\crcr\unhbox0}}} \def\cprime{$'$}

\end{document}